\documentclass[11pt]{article}
\usepackage{amsmath}
\usepackage{amssymb}
\usepackage{graphicx}
\usepackage{amsthm}
\usepackage{verbatim}
\usepackage{color}
\usepackage{enumerate}
\usepackage{amsfonts}
\usepackage[all]{xy}
\usepackage{fancyhdr}
\usepackage{enumerate}
\theoremstyle{plain}

\newtheorem{theorem}{Theorem}[section]
\newtheorem{lemma}[theorem]{Lemma}

\theoremstyle{remark}
\newtheorem{remark}[theorem]{Remark}

\theoremstyle{definition}
\newtheorem{definition}[theorem]{Definition}
\newtheorem{example}[theorem]{Example}

\newtheorem{convention}[theorem]{Convention}

\begin{document}
\title{Triangulation of orbifolds with corners}
\author{Hao Yu}\date{}
\maketitle
\begin{abstract}
We prove that every compact orbifold with corners admits a triangulation compatible with all its corners: for every $k\ge1$ the image of every connected component of the $k$-corners is the image of a subcomplex, and every open simplex lies in the set of points of a single depth. Faces need not be embedded, and the local groups may permute boundary components. The proof turns the boundary hyperplanes of linear charts into mirrors by means of the squaring map, which, unlike the folding $y\mapsto|y|$, is compatible with smooth changes of charts, and then applies the triangulation theorem for orbit spaces of proper Lie groupoids.
\end{abstract}

\section{Introduction}
An orbifold with corners is locally the quotient of an open subset of $R^n_k=[0,\infty)^k\times\mathbb{R}^{n-k}$ by a finite group acting smoothly and preserving the corner structure; the group may permute the boundary hyperplanes. Such spaces occur, for instance, as compactified moduli spaces of bordered Riemann surfaces, whose codimension one faces need not be embedded and whose local groups may permute boundary components. We prove:

\begin{theorem}\label{thm_intro}
Every compact orbifold with corners $P$ admits a triangulation $tri:|L|\to|P|$ such that
\begin{enumerate}[(a)]
\item every open simplex of $L$ is mapped into the set of points of a single depth, and
\item for every $k\ge1$ and every connected component $F$ of the $k$-corners $C_k(P)$, $i_k(F)$ is the image of a subcomplex.
\end{enumerate}
For $k=1$, (b) says that $i_P(F)$ is the image of a subcomplex for every codimension one face $F$. In particular, $|P|$ is a compact polyhedron.
\end{theorem}

Theorem \ref{thm_intro} makes combinatorial tools available for such spaces, for instance for compactified moduli spaces of bordered Riemann surfaces, whose faces need not be embedded: simplicial and cellular chain models of $|P|$ relative to $i_P(\partial P)$ or to unions of faces, simplicial representatives of fundamental classes (for oriented $P$, with rational coefficients), and simplicial approximation of continuous maps.

Here $i_P:\partial P\to P$ is the boundary in the sense of Joyce \cite{jo} (Convention \ref{conv_faces}); faces need not be embedded. For orbifolds without boundary, triangulability of the underlying space follows from the triangulation of orbit spaces of proper Lie groupoids \cite[Corollary 7.2]{ppt}, which rests on the existence of control data and on the triangulation of abstract prestratified spaces \cite{gor78,v}. Moerdijk and Pronk \cite[Prop.~1.2.1]{mp99} regard such a triangulation as well known, citing Yang \cite{y} and Goresky \cite{gor78}; Yang's proof relies on a triangulation result of Cairns \cite{css} for regular locally polyhedral spaces which is incorrect, as recalled in the introduction of \cite{v}.

The natural approach to corners is doubling: by Lemma \ref{lem_doubling}, the folding $(\eta,l)\mapsto(|\eta_1|,\dots,|\eta_k|,l)$ identifies a linear corner chart $R^n_k/G$ with $\mathbb{R}^n/\hat G$, $\hat G=\mathbb{Z}_2^k\rtimes G$, turning the boundary hyperplanes into mirrors. This identification is a homeomorphism, but it is not compatible with smooth changes of charts: the corner-preserving diffeomorphism $y\mapsto y+y^2$ of $[0,\infty)$ lifts under the folding to $\eta\mapsto\eta+\eta|\eta|$, which is only $C^1$. Hence the folded charts do not define a smooth structure on $|P|$, and the theory of Whitney stratifications and control data cannot be applied to them directly. The squaring map $(\eta,l)\mapsto(\eta_1^2,\dots,\eta_k^2,l)$ does not have this defect: by Hadamard's lemma, corner-preserving embeddings lift to smooth embeddings (Step 2 of the proof of Theorem \ref{thm_triangulation}). This produces an orbifold $P'$ without boundary with $|P'|=|P|$, to which \cite{ppt} applies, and the faces of $P$ are then read off from the orders of the local groups of $P'$. Even when all local groups of $P$ are trivial, i.e.\ for manifolds with corners with possibly non-embedded faces, $P'$ is a reflection orbifold, so this case, too, is reduced to orbifolds.

The case $k=1$ of Theorem \ref{thm_intro}(b) was stated in earlier versions of \cite{yu_ght}, with an incomplete proof based on the folding. It is not needed for the comparison theorem proved there.

In Section \ref{sec_orb} we recall orbifolds with corners, their boundaries and linear charts; Section \ref{sec_tri} contains the proof.

\section{Orbifolds with corners}\label{sec_orb}
We write $R^n=\mathbb{R}^n$, $Z_2=\mathbb{Z}/2$, $R^n_k=[0,\infty)^k\times\mathbb{R}^{n-k}$, $R^n_+=R^n_n$, and $R^n_{+k}\subseteq R^n_+$ for the set of points with exactly $k$ coordinates equal to $0$, i.e.\ the points of depth $k$. Smooth maps on subsets of $R^n_k$ are understood to be smooth up to the corners, i.e.\ locally restrictions of smooth maps defined on open subsets of $\mathbb{R}^n$. All orbifolds are assumed to be \textit{effective}: each local group acts effectively on its chart.

\begin{definition}\label{orbifold}
(Orbifolds).
Let $X$ be a topological space, and fix $n \geq 0$.
\begin{itemize}
\item An $n$-dimensional orbifold chart on $X$ is given by a connected open subset
$\tilde{U}\subseteq R^n$, a finite group $G$ acting smoothly on $\tilde{U}$, 
and a map $\varphi:\tilde{U}\rightarrow X$ such that $\varphi$ is $G$-invariant and induces a homeomorphism of $\tilde{U}/G$  onto
an open subset $U \subseteq  X$.
\item An embedding $\lambda : (\tilde{U} , G, \varphi  ) \hookrightarrow (\tilde{V} ,H, \psi)$ between two such charts is a
smooth embedding $\lambda : \tilde{U} \hookrightarrow \tilde{V}$
 with $\psi \lambda = \varphi$.  An orbifold atlas on $X$ is a family $\mathcal{U} = {(\tilde{U} , G, \varphi)}$ of such charts, which
cover $X$ and are locally compatible: given any two charts $(\tilde{U} , G, \varphi )$ for
$U = \varphi(\tilde{U}) \subset X$ and $(\tilde{V},H,\psi  )$ for $V \subseteq X$, and a point $x \in U \cap V$ , there
exists an open neighbourhood $W \subseteq U \cap V$ of $x$ and a chart $(\tilde{W},K,\mu  )$ for $W$
such that there are embeddings $(\tilde{W},K,\mu  ) \hookrightarrow (\tilde{U}, G, \varphi  )$ and $(\tilde{W},K,\mu  )\hookrightarrow (\tilde{V},H,\psi)$.
\item An atlas $\mathcal{U}$ is said to refine another atlas $\mathcal{V}$ if for every chart in $\mathcal{U}$ there
exists an embedding into some chart of $\mathcal{V}$. Two orbifold atlases are said to be
equivalent if they have a common refinement.
\end{itemize}
An orbifold of dimension $n$, or $n$-orbifold, is a paracompact Hausdorff space $X$ equipped with an equivalence class $\mathcal{U}$ of $n$-dimensional orbifold atlases.
\end{definition}

\begin{definition}\label{orbi_corners}
(Orbifolds with corners).
Let $X$ be a topological space, and fix $n \geq 0$.
\begin{itemize}
\item A generalized $n$-dimensional orbifold chart on $X$ is given by a connected open subset
$\tilde{U} \subseteq R^n_+$, a finite group $G$ acting smoothly on $\tilde{U}$, 
and a map $\varphi :\tilde{U} \rightarrow X$ such that $\varphi$ is $G$-invariant, $G$ maps each $\tilde{U}_k := \tilde{U} \cap R^n_{+k}$ to itself (it may permute the connected components of $\tilde U_k$), and $\varphi$ induces a homeomorphism of $\tilde{U}/G$  onto
an open subset $U \subseteq X$.
\item An embedding $\lambda : (\tilde{U} , G, \varphi  ) \hookrightarrow (\tilde{V} ,H,\psi  )$ between two such charts is a
diffeomorphism $\lambda:\tilde{U}\to\lambda(\tilde U)$ onto an open subset of $\tilde{V}$ such that $\lambda$ and $\lambda^{-1}$ are smooth up to the corners (locally restrictions of smooth maps between open subsets of $R^n$), with $\lambda(\tilde{U}_k)\subseteq\tilde{V}_k$ for all $k$, and $\psi \lambda = \varphi$.  An orbifold atlas on $X$ is a family $\mathcal{U} = {(\tilde{U} , G, \varphi  )}$ of such charts, which
cover $X$ and are locally compatible: given any two charts $(\tilde{U} , G, \varphi  )$ for
$U = \varphi(\tilde{U}) \subseteq X$ and $(\tilde{V} ,H, \psi )$ for $V \subseteq X$ and a point $x \in U \cap V$ , there
exists an open neighbourhood $W \subseteq U \cap V$ of $x$ and a chart $(\tilde{W},K,\mu  )$ for $W$
such that there are embeddings $(\tilde{W},K,\mu  ) \hookrightarrow (\tilde{U} , G, \varphi  )$ and $(\tilde{W},K,\mu  ) \hookrightarrow (\tilde{V} ,H,\psi  )$.
\item An atlas $\mathcal{U}$ is said to refine another atlas $\mathcal{V}$ if for every chart in $\mathcal{U}$ there
exists an embedding into some chart of $\mathcal{V}$. Two generalized orbifold atlases are said to be
equivalent if they have a common refinement.
\end{itemize}
An orbifold with corners of dimension $n$, or $n$-orbifold with corners, is a paracompact Hausdorff space $X$ equipped with an equivalence class $\mathcal{U}$ of $n$-dimensional generalized orbifold atlases.
\end{definition}
\begin{remark}
In Definition \ref{orbi_corners} the local groups may permute the connected components of $\tilde U_k$, as in Convention \ref{conv_faces} and Lemma \ref{lem_doubling}; this is needed, for instance, for moduli spaces of bordered Riemann surfaces, where automorphisms may permute boundary components. Requiring each component to be $G$-invariant gives a strictly smaller class: $[0,\infty)^3/\mathbb{Z}_3$, with $\mathbb{Z}_3$ permuting the coordinates cyclically, has a single ray of points of depth two emanating from the image of $0$, whereas every model in which the local group preserves the components has three. 
\end{remark}

\begin{convention}[boundary]\label{conv_faces}
We use the boundary of a manifold with corners in the sense of Joyce \cite[Definitions~2.6 and 2.10]{jo}. As a set, $\partial M$ consists of pairs $(x,\beta)$, where $x\in M$ and $\beta$ is a local boundary component of $M$ at $x$. It is a manifold with corners of dimension $\dim M-1$, and the map $i_M:\partial M\to M$, $(x,\beta)\mapsto x$, is smooth \cite[Theorem~3.4(iv)]{jo}. A point of depth $l$ (i.e.\ one that is sent to the origin by a chart modelled on $R^n_l$) has exactly $l$ preimages under $i_M$. 

For an orbifold with corners $Q$, the \textit{depth} of $x\in|Q|$ is the depth of any preimage of $x$ in a chart; it is well defined since local groups and embeddings of charts preserve depth. Locally $Q=\tilde U/G$, and we put $\partial(\tilde U/G):=(\partial\tilde U)/G$, where $g\in G$ acts by $(x,\beta)\mapsto(gx,g(\beta))$. Since $G$ may permute the connected components of $\partial\tilde U$, the local models are $(\tilde C,G_{\tilde C})$, where $\tilde C$ runs over these components and $G_{\tilde C}$ is the stabiliser of $\tilde C$. These local models glue to an orbifold with corners $\partial Q$ together with a map $i_Q:\partial Q\to Q$. A point of depth $l\ge1$ now has between $1$ and $l$ preimages under $i_Q$, since the local group may permute the local boundary components. The orbifold $\partial Q$ is again effective, since a finite order linear map fixing a boundary hyperplane pointwise and preserving the half-space is the identity. The \textit{codimension one faces} of $Q$ are the connected components of $\partial Q$. We say that $Q$ has \textit{embedded faces} if $i_Q$ is injective on each of them. Simplices and their products have embedded faces. The teardrop \cite[Remark~2.11]{jo} does not (Example \ref{ex_teardrop}), and neither do many moduli spaces of curves. We do not assume embedded faces.

More generally, for $k\ge1$ the \textit{$k$-corners} $C_k(M)$ of a manifold with corners $M$ are, as a set, the pairs $(x,\{\beta_1,\dots,\beta_k\})$ with $\beta_1,\dots,\beta_k$ distinct local boundary components of $M$ at $x$, with the map $i_k(x,\{\beta_j\})=x$ \cite{jo}. In a chart modelled on $R^n_l$ they are $\bigsqcup_{I}\{y_i=0,\ i\in I\}$ over the subsets $I\subseteq\{1,\dots,l\}$ with $|I|=k$, and $C_1(M)=\partial M$, $i_1=i_M$. For an orbifold with corners $Q$, locally $Q=\tilde U/G$, put $C_k(\tilde U/G):=C_k(\tilde U)/G$, where $g$ acts by $(x,\{\beta_j\})\mapsto(gx,\{g(\beta_j)\})$. Embeddings of charts are diffeomorphisms onto open subsets, so they induce bijections of local boundary components \cite{jo}, and, since $g\lambda(\tilde W)\cap\lambda(\tilde W)\ne\emptyset$ for an embedding $\lambda$ of $(\tilde W,K)$ into $(\tilde U,G)$ implies that $g$ lies in the image of $K$, the induced maps $C_k(\tilde W)/K\to C_k(\tilde U)/G$ are injective open maps. So these local quotients glue to a space $C_k(Q)$ with a continuous map $i_k:C_k(Q)\to|Q|$; we use only this space and its connected components, not an orbifold structure on it. It is Hausdorff: points with different images are separated through $i_k$, and the preimage of a point $x$ lies in $C_k(\tilde U)/G$ for a linear chart centred at $x$, a disjoint union of convex pieces permuted by $G$, in which the unions of the $G$-orbits of pieces are disjoint open sets. It is compact if $Q$ is. The set $i_k(C_k(Q))$ consists of the points of depth at least $k$, and a point of depth $l\ge k$ has between $1$ and $\binom lk$ preimages under $i_k$. For $k=1$, $C_1(Q)$ is $\partial Q$ as a space, and its components are the codimension one faces.
\end{convention}

\begin{example}\label{ex_teardrop}
Let $P$ be the teardrop, a closed disc with one corner point $c$, regarded as an orbifold with corners with trivial local groups, and let $C$ be its boundary circle. The point $c$ has depth two and the other points of $C$ have depth one, so $i_P$ maps $i_P^{-1}(C\setminus\{c\})$ bijectively onto the open arc $C\setminus\{c\}$, and the two local boundary components at $c$ are the germs of the two ends of this arc. Hence $\partial P$ has a single face $F$, which is diffeomorphic to $[0,1]$, and $i_P$ maps both endpoints of $F$ to $c$; thus $i_P$ is not injective on $F$, and $i_P(F)=C$ is the whole boundary circle. By contrast, the square $[0,1]^2$ has four faces, each mapped diffeomorphically onto an edge, so its faces are embedded. For orbifolds, the local groups produce a related phenomenon already locally. Let $Z_2$ act on $R^3_2$, with coordinates $(y_1,y_2,l)$, by interchanging $y_1$ and $y_2$. Since the two boundary hyperplanes $\{y_1=0\}$ and $\{y_2=0\}$ are interchanged, every point of depth two of $Q=R^3_2/Z_2$, i.e.\ every point of the image of the corner $\{y_1=y_2=0\}$, has a single preimage under $i_Q$, that is, a single local boundary component.
\end{example}

We shall use the following linearization of corner charts. The doubled group in (2) is in general a semidirect product, since the local group may permute the boundary hyperplanes.
\begin{lemma}[Doubling lemma]\label{lem_doubling}
Let $G$ be a finite group acting smoothly on a $G$-invariant open neighbourhood $U$ of $0$ in $R^n_k$, preserving the corner structure and fixing $0$. Then:
\begin{enumerate}
\item After shrinking $U$ and applying a $G$-equivariant, corner-preserving diffeomorphism onto a neighbourhood of $0$, there is a linear isomorphism $R^n\cong R^k\oplus R^{n-k}$, $x=(y,l)$, carrying $R^n_k$ to $\{y_1,\dots,y_k\ge 0\}$, such that $G$ acts linearly by $g(y,l)=(P_g\,y,\ A_g\,l)$ with $P_g$ a permutation matrix.
\item Let $\Gamma\cong Z_2^k$ be the group generated by the reflections $r_i: y_i\mapsto -y_i$. Then $G$ normalizes $\Gamma$, and $\hat G:=\Gamma\rtimes G$ acts linearly on $R^n$. The inclusion $R^n_k\subset R^n$ induces a homeomorphism $R^n_k/G\cong R^n/\hat G$.

\end{enumerate}
\end{lemma}
\begin{proof}
(1) Since $g$ preserves the corner structure, $dg_0$ preserves the tangent cone $R^n_k$ and permutes the boundary hyperplanes $\{x_i=0\}$, $i\le k$. Put $\varphi(x)=\frac{1}{|G|}\sum_{g\in G}(dg_0)^{-1}g(x)$. Each summand maps $U\cap\{x_i=0\}$ into $\{x_i=0\}$ and $U\cap\{x_i>0\}$ into $\{x_i>0\}$, and these sets are convex, so $\varphi$ preserves the corner structure. Moreover $d\varphi_0=\mathrm{id}$ and $\varphi\circ g=dg_0\circ\varphi$. By Hadamard's lemma $x_i\circ\varphi=x_i\,c_i$ with $c_i(0)=1$, $i\le k$. Extend $\varphi$ smoothly to a neighbourhood of $0$ in $R^n$ with $i$-th component $x_i\tilde c_i$, where $\tilde c_i>0$ extends $c_i$. By the inverse function theorem this extension is a diffeomorphism near $0$, and it maps $R^n_k$ onto $R^n_k$ near $0$, preserving depth in both directions. Hence $\varphi$ is a corner-preserving diffeomorphism near $0$ that linearizes the action. Now choose a $G$-invariant inner product. The lineality space $L=\{x_1=\dots=x_k=0\}$ of the cone $R^n_k$ is $G$-invariant, hence so is $L^\perp$, and $R^n_k=C\times L$ with $C\subset L^\perp$ a simplicial cone. The group $G$ permutes the extremal rays of $C$; if $w_i$ denotes the unit vector on the $i$-th ray, then $g w_i=w_{\sigma_g(i)}$ exactly. The coordinates $y_i$ with respect to the basis $(w_i)$ of $L^\perp$, together with linear coordinates $l$ on $L$, have the required properties.

(2) $g r_i g^{-1}=r_{\sigma_g(i)}$, so $G$ normalizes $\Gamma$. The map $(y,l)\mapsto(|y_1|,\dots,|y_k|,l)$ $G$-equivariantly identifies $R^n/\Gamma$ with $R^n_k$; hence $R^n/\hat G=(R^n/\Gamma)/G\cong R^n_k/G$.

\end{proof}

\section{The triangulation}\label{sec_tri}
By a \textit{triangulation} of a compact space $Y$ we mean a homeomorphism $tri: |L| \rightarrow Y$ from the underlying space of a finite simplicial complex $L$. It is \textit{subordinate} to a stratification $\mathcal{S}$ of $Y$ if $tri^{-1}(R)$ is a union of open simplices of $L$ for every stratum $R \in \mathcal{S}$. If $Y=|P|$ for an orbifold with corners $P$, we say that the triangulation is \textit{compatible with the corners} if every open simplex of $L$ is mapped into the set of points of a single depth, and $i_k(F)$ is the image of a subcomplex for every $k\ge1$ and every connected component $F$ of $C_k(P)$ (Convention \ref{conv_faces}).

\begin{theorem}[Theorem \ref{thm_intro}]\label{thm_triangulation}
Every compact orbifold with corners $P$ admits a triangulation compatible with its corners. In particular, $|P|$ is a compact polyhedron.
\end{theorem}
\begin{proof}
We turn the boundary hyperplanes of the charts of $P$ into mirrors. This gives an orbifold $P'$ \emph{without} boundary with $|P'|=|P|$, to which the triangulation theorem for orbit spaces of proper Lie groupoids applies; the faces of $P$ are then read off from the local groups of $P'$. The proof uses no ambient manifold.

\textit{Step 1: linear charts.} Let $x\in|P|$ have depth $d$. Restrict a chart to a small neighbourhood of a preimage $\tilde x$ of $x$ that is invariant under the stabiliser of $\tilde x$, translate and permute coordinates, and apply Lemma \ref{lem_doubling}(1), choosing the coordinates $l$ orthonormal for a $G$-invariant inner product. This gives a chart $(\tilde U,G,\varphi)$ with $\varphi(0)=x$, coordinates $(y,l)\in R^d\times R^{n-d}$, $g(y,l)=(P_gy,A_gl)$ with $P_g$ a permutation matrix and $A_g$ orthogonal, and $\tilde U=\{0\le y_i<\varepsilon,\ |l|<\varepsilon\}$. (A translation of the $l$-coordinates identifies $\tilde U$ with an open subset of $R^n_+$, as Definition \ref{orbi_corners} requires.) We call such a chart \emph{linear, centred at $x$}. The group $G$ acts effectively on $\tilde U$: an element acting trivially on $\tilde U$ is a finite-order diffeomorphism of the original (connected) chart that is the identity on a non-empty open subset, hence is the identity, and the local groups of $P$ act effectively. Restrictions of charts to smaller invariant neighbourhoods are charts, and embeddings compose with such restrictions and with linearizing diffeomorphisms. Hence for any two charts whose images contain $x$ there is an arbitrarily small linear chart centred at $x$ embedding into both. By Definition \ref{orbi_corners}, embeddings of charts are diffeomorphisms onto open subsets, smooth up to the corners, and they preserve depth.

\textit{Step 2: the orbifold $P'$.} For a linear chart put $s(\eta,l):=(\eta_1^2,\dots,\eta_d^2,l)$, $\hat U:=s^{-1}(\tilde U)=\{|\eta_i|<\varepsilon^{1/2},\ |l|<\varepsilon\}$ and $\hat\varphi:=\varphi\circ s$. Let $\hat G=\Gamma\rtimes G$, $\Gamma\cong Z_2^d$, act linearly on $R^n$ as in Lemma \ref{lem_doubling}(2). Then $s\circ\hat g=g\circ s$ for $\hat g\in\hat G$ with image $g\in G$, and $s$ induces $\hat U/\Gamma\cong\tilde U$. Hence $\hat\varphi$ induces a homeomorphism $\hat U/\hat G\cong\tilde U/G$. The action is effective: if $\gamma g$ acts trivially, then $g$ acts trivially on $s(\hat U)=\tilde U$, so $g=1$ and then $\gamma=1$.

Let $\lambda\colon\tilde W\to\tilde U$ be an embedding of a linear chart $(\tilde W,K,\mu)$ centred at a point of depth $d'$. Let $J$ be the set of $j\le d$ with $y_j(\lambda(0))=0$. Then $|J|=d'$, and after shrinking $\tilde W$ we have $y_j\circ\lambda>0$ for $j\notin J$. For $i\le d'$ the connected set $\{y_i=0,\ y_m>0\ (m\le d',\ m\neq i)\}\cap\tilde W$ consists of points of depth one, so $\lambda$ maps it, and hence, by continuity, $\{y_i=0\}\cap\tilde W$, into a single hyperplane $\{y_{\sigma(i)}=0\}$ with $\sigma(i)\in J$. Since $d\lambda_0$ is invertible, it cannot map two different hyperplanes $\{y_i=0\}$, $\{y_{i'}=0\}$ into one hyperplane, so $\sigma$ is a bijection onto $J$. Since $\lambda$ preserves depth, $y_{\sigma(i)}\circ\lambda$ vanishes exactly on $\{y_i=0\}$. As $\tilde W$ is convex and stable under $y_i\mapsto ty_i$, $t\in[0,1]$, Hadamard's lemma gives $y_{\sigma(i)}\circ\lambda=y_i\,a_i$ with $a_i$ smooth. We have $a_i>0$ off $\{y_i=0\}$. On $\{y_i=0\}$ we have $a_i=\partial_{y_i}(y_{\sigma(i)}\circ\lambda)\neq0$, since $d\lambda$ is invertible and maps $T\{y_i=0\}$ into $\ker dy_{\sigma(i)}$; hence $a_i>0$ there too. Define $\hat\lambda\colon\hat W\to\hat U$ by
\begin{equation*}
\eta'_{\sigma(i)}=\eta_i\,(a_i\circ s)^{1/2}\ (i\le d'),\qquad \eta'_j=(y_j\circ\lambda\circ s)^{1/2}\ (j\notin J),\qquad l'=l\circ\lambda\circ s .
\end{equation*}
Then $\hat\lambda$ is smooth and $s\circ\hat\lambda=\lambda\circ s$, hence $\hat\varphi\circ\hat\lambda=\hat\mu$. Moreover $\mathrm{sign}\,\eta'_{\sigma(i)}=\mathrm{sign}\,\eta_i$. The map $\hat\lambda$ is injective: if $\hat\lambda(p)=\hat\lambda(q)$, then $s(p)=s(q)$ because $\lambda$ is injective, so $p$ and $q$ differ only by signs of the $\eta_i$, and these signs agree because $\hat\lambda$ preserves them. It is also a local diffeomorphism. Let $p\in\hat W$; we adapt the construction to $\lambda^{-1}$ on a small box $B$ around $q=\lambda(s(p))$, writing $s$ for the squaring maps of both charts. For $i\le d'$ with $y_i(s(p))=0$, Hadamard's lemma on $B$ (which is convex and stable under $y_{\sigma(i)}\mapsto ty_{\sigma(i)}$) gives $y_i\circ\lambda^{-1}=y_{\sigma(i)}b_i$ with $b_i>0$, and we put $\kappa_i=\eta'_{\sigma(i)}(b_i\circ s)^{1/2}$; for the other $i\le d'$ we put $\kappa_i=\pm(y_i\circ\lambda^{-1}\circ s)^{1/2}$ with the sign of $\eta_i(p)$; the $l$-components are $l\circ\lambda^{-1}\circ s$. This gives a smooth $\kappa$ with $s\circ\kappa=\lambda^{-1}\circ s$ and $\kappa(\hat\lambda(p))=p$. Then $h=\kappa\circ\hat\lambda$ satisfies $s\circ h=s$, each coordinate $h_i$ is $\eta_i$ times a positive smooth function or $\pm\eta_i$ with constant sign, and $h(p)=p$; hence $h=\mathrm{id}$ near $p$, and likewise $\hat\lambda\circ\kappa=\mathrm{id}$ near $\hat\lambda(p)$. So $\hat\lambda$ is an embedding of charts. By Step 1, the charts $(\hat U,\hat G,\hat\varphi)$ cover $|P|$ and are locally compatible, so they form an orbifold atlas in the sense of Definition \ref{orbifold}. This defines a compact effective orbifold $P'$ without boundary with $|P'|=|P|$, whose local group at $\hat\varphi(\eta)$ is the stabiliser $\hat G_\eta$.

The squaring is essential. With the folding $y=|\eta|$ of Lemma \ref{lem_doubling}(2), the corner-preserving diffeomorphism $y\mapsto y+y^2$ of $[0,\infty)$ lifts to $\eta\mapsto\eta+\eta|\eta|$, which is only $C^1$; with $y=\eta^2$ it lifts to $\eta\mapsto\eta(1+\eta^2)^{1/2}$. For the same reason, charts of $|P|$ built from the folding, such as those given by invariant polynomials of $\hat G$, are in general not compatible under corner-preserving changes of coordinates: in the chart $u=y^2$ the map above becomes $u\mapsto u+2u^{3/2}+u^2$, which is not smooth at $0$. The construction is local, so faces need not be embedded, and the permutations $P_g$ and $\sigma$ account for local groups and embeddings that permute boundary hyperplanes. For the teardrop (Example \ref{ex_teardrop}), $P'$ is a disc whose boundary circle is a mirror, with local group $Z_2^2$ at the corner point.

\textit{Step 3: triangulating $P'$.} The orbifold $P'$ is represented by a proper \'etale Lie groupoid with orbit space $|P|$ (\cite{mp}, \cite[\S3]{mo}). By \cite[Cor.~5.4]{ppt}, this orbit space, stratified by orbit types, carries a smooth structure satisfying Whitney's condition (B) in its singular charts, and admits smooth control data. These control data come from \cite[Thm.~2.9]{ppt}, cf.\ \cite[Prop.~7.1]{mat70} and \cite[Thm.~3.6.9]{pf01}, and do not require an ambient manifold. By \cite[Thm.~7.1, Cor.~7.2]{ppt}, which rest on \cite[5.~Proposition]{gor78} and \cite[3.7.~Corollary]{v}, there is a triangulation $tri\colon|L|\to|P|$ such that $tri^{-1}(\bar S)$ is a subcomplex for every stratum $S$ of the decomposition of \cite[Thm.~5.7(ii)]{ppt}. (In \cite[\S7]{ppt} this is phrased as $tri^{-1}(S)$ being a subcomplex, cf.\ \cite[\S3]{gor78}; since strata need not be closed, we read it as a statement about their closures.) The complex $L$ is finite since $|L|$ is compact. The strata are the connected components of the sets of points of $|P'|$ whose local groups lie in a given isomorphism class (\cite[Def.~5.6(i), Thm.~5.7(ii),(iii), Prop.~5.12]{ppt}; for orbifolds this is the usual stratification, \cite[Rem.~5.5(ii)]{ppt}), so the order $n(S)$ of the local group is constant on each stratum $S$.

We claim that $R\neq S$ and $R\cap\bar S\neq\emptyset$ imply $R\subseteq\bar S$ and $n(R)>n(S)$. The inclusion is \cite[Def.~2.1 (DC2)]{ppt}. For the inequality, take $r=\hat\varphi(\eta)\in R$, put $H=\hat G_\eta$, and use the chart of $P'$ given by a small $H$-invariant ball around $\eta$ in $\hat U$, on which $H$ acts by affine maps fixing $\eta$. Points near $r$ have local groups isomorphic to subgroups of $H$. Those whose local group has order $|H|$ are the images of the fixed-point set of $H$ in the ball, which is connected and contained in $R$. Since $r\in\bar S$ and $S\cap R=\emptyset$, points of $S$ near $r$ have local groups of order less than $|H|=n(R)$. Consequently, if $R\subseteq\bar S\setminus S$, every stratum meeting $\bar R$ has order at least $n(R)>n(S)$, so $\bar R\cap S=\emptyset$. Hence $tri^{-1}(S)=tri^{-1}(\bar S)\setminus\bigcup_{R\subseteq\bar S\setminus S}tri^{-1}(\bar R)$ is a union of open simplices, i.e.\ $tri$ is subordinate to the orbit type stratification, and the order of the local group of $P'$ is constant on each open simplex of $L$.

\textit{Step 4: depth.} Let $(\tilde U,G,\varphi)$ be a linear chart centred at a point $x$ of depth $d$, with doubled chart $(\hat U,\hat G,\hat\varphi)$, and let $r_i\in\Gamma$ be the reflection $\eta_i\mapsto-\eta_i$. For $\eta\in\hat U$ the point $\hat\varphi(\eta)$ has depth $\#\{i:\eta_i=0\}=\#\{i:r_i\in\hat G_\eta\}$, since $r_i$ fixes $\eta$ exactly when $\eta_i=0$. We claim that the depth is locally constant on every set on which the order of the local group of $P'$ is constant. Let $z=\hat\varphi(\eta)$, $H=\hat G_\eta$, and use a small $H$-invariant ball $B$ around $\eta$, so that $\hat G_{\eta'}\subseteq H$ for $\eta'\in B$. If $\hat\varphi(\eta')$ has local group of order $|H|$, then $\hat G_{\eta'}=H$, so $\{i:r_i\in\hat G_{\eta'}\}=\{i:r_i\in H\}$ and $\hat\varphi(\eta')$ has the depth of $z$. By Step 3 the order of the local group is constant on each open simplex, and open simplices are connected, so the depth is constant on each open simplex of $L$. This is (a).

\textit{Step 5: the corners.} Fix $k\ge1$. For $z\in|P|$ let $\Phi_k(z)$ be the set of connected components $F$ of $C_k(P)$ with $z\in i_k(F)$. We claim that $\Phi_k$ is locally constant on every set on which the order of the local group of $P'$ is constant. Let $(\tilde U,G,\varphi)$ be a linear chart centred at a point $x$ of depth $d$, and let $z=\hat\varphi(\eta)$ with $|\hat G_\eta|=|\hat G|$; these are the points near $x$ whose local group has the order of that of $x$. Then $\hat G_\eta=\hat G\supseteq\Gamma$, so $\eta_1=\dots=\eta_d=0$. Thus $z=\varphi(0,l)$ lies on the $d$ hyperplanes $\{y_i=0\}$ and has depth $d$, and its local boundary components are, as at $x$, the germs of $\{y_i=0\}$, $i\le d$. Near $x$, $C_k(P)=C_k(\tilde U)/G$ with $C_k(\tilde U)=\bigsqcup_{|I|=k}\tilde U\cap\{y_i=0,\ i\in I\}$ (Convention \ref{conv_faces}). Each piece is convex, hence connected, so for every $I$ the points of $C_k(P)$ given by $I$ at $z$ and at $x$ lie in the same component of $C_k(P)$; and every point of $i_k^{-1}(z)$ and of $i_k^{-1}(x)$ is given by some $I$. Thus $\Phi_k(z)=\Phi_k(x)$. For $k=1$ this is the statement for the codimension one faces.

Open simplices are connected, so by Step 3 $\Phi_k\circ tri$ is constant on each of them, and $tri^{-1}(i_k(F))$ is a union of open simplices for every component $F$ of $C_k(P)$. It is closed, since $i_k(F)$ is compact: $C_k(P)$ is compact, being covered by the images of finitely many sets $\overline{\tilde U'}\cap\{y_i=0,\ i\in I\}$ for slightly smaller linear charts $\tilde U'$, and its component $F$ is closed in it. A closed union of open simplices is a subcomplex. This is (b), and $|P|\cong|L|$ is a compact polyhedron.

For the teardrop (Example \ref{ex_teardrop}), $C_2(P)$ is the single point over the corner $c$, so $c$ is a vertex of $L$ (this also follows from (a), since $c$ is the only point of depth two). This does not follow from the statement for codimension one faces alone, since the only face has image the whole boundary circle.
\end{proof}

\section{Tool Use}
The theorem and the strategy of the proof (doubling the corners into mirrors, Whitney regularity, and Verona's triangulation theorem) are due to the author, from earlier versions of arXiv:2004.07698. The AI assistant checked the validity of the proofs. The author takes full responsibility for the content.

\noindent Hao Yu, {\it Independent researcher, Beijing, China.}\\
{\tt dustless2014@163.com}
\end{document}